\documentclass[reqno,letterpaper,11pt]{amsart}

\usepackage{hyperref}

\usepackage{comment}
\usepackage{enumerate}
\usepackage{dsfont}
\usepackage{mathtools}
\usepackage{amsmath}
\usepackage{amsfonts,amssymb,amsmath,latexsym,wasysym,mathrsfs,bbm,stmaryrd}
\usepackage[all]{xy}
\usepackage[usenames]{color}
\usepackage{epsfig}

\usepackage{amsthm}
\usepackage{thmtools}

\newtheorem{theorem}{Theorem}[section]

\newtheorem{corollary}[theorem]{Corollary}

\newtheorem{lemma}[theorem]{Lemma}

\theoremstyle{remark}

\newtheorem{remark}[theorem]{Remark}
\numberwithin{equation}{section}
\numberwithin{equation}{section}

\usepackage[mathcal]{euscript}
\usepackage{times}

\def\R{\mathbb R}

\def\N{\mathbb N}
\def\E{\mathbb E}
\def\e{\epsilon}

\def\1{\mathbbm{1}}

\newcommand{\mx}[1]{\mathbf{#1}}

\long\def\symbolfootnote[#1]#2{\begingroup%
\def\thefootnote{\fnsymbol{footnote}}\footnote[#1]{#2}\endgroup}

\begin{document}

\title{Random matrices with log-range correlations, and log-Sobolev inequalities}
\author{Todd Kemp}
\thanks{Kemp supported in part by NSF CAREER Award DMS-1254807}
\address{Department of Mathematics\\
University of California, San Diego \\
La Jolla, CA 92093-0112}
\email{tkemp@math.ucsd.edu}
\author{David Zimmermann}
\address{Department of Mathematics\\
  University of Chicago \\
  5734 S.\ University Avenue \\
  Chicago, IL 60637}
\email{dz@math.uchicago.edu}
\maketitle

\begin{abstract}
Let $X_N$ be a symmetric $N\times N$ random matrix whose $\sqrt{N}$-scaled centered entries are uniformly square integrable.  We prove that if the entries of $X_N$ can be partitioned into independent subsets each of size $o(\log N)$, then the empirical eigenvalue distribution of $X_N$ converges weakly to its mean in probability.  This significantly extends the best previously known results on convergence of eigenvalues for matrices with correlated entries (where the partition subsets are blocks and of size $O(1)$.)  we prove this result be developing a new log-Sobolev inequality, generalizing the first author's introduction of mollified log-Sobolev inequalities: we show that if $\mx{Y}$ is a bounded random vector and $\mx{Z}$ is a standard normal random vector independent from $\mx{Y}$, then the law of $\mx{Y}+t\mx{Z}$ satisfies a log-Sobolev inequality for all $t>0$, and we give bounds on the optimal log-Sobolev constant.

%We give upper bounds for optimal constants in logarithmic Sobolev inequalities for convolutions of compactly supported measures on $\mathbb{R}^d$ with a Gaussian measure. We investigate tightness of these bounds by examining some examples. We then revisit and improve (by significantly weakening hypotheses) the statement and proof in \cite{Zi13} that, under certain hypotheses, the empirical law of eigenvalues of a sequence of symmetric random matrices converges to its mean in probability. In particular, the universality theorem holds even when there is some mild long range dependence between the entries.
\end{abstract}

\section{Introduction}

Random matrix theory is primarily interested the convergence of statistics associated to the eigenvalues (or singular values) of $N\times N$ matrices whose entries are random variables with a prescribed joint distribution.    The field was begun by Wigner in \cite{Wi55,Wi58}.  The bulk of the modern field, devoted to the (mostly settled) {\em universality program}, is concerned with two families of symmetric random matrix ensembles:
\begin{itemize}
\item {\em Wigner ensembles} $X_N$: the entries of $\sqrt{N}X_N$ are i.i.d.\ random variables (modulo the symmetry constraint) with sufficiently many finite moments; and
\item {\em Invariant ensembles} $X_N$ the joint law of entries has a density with respect to the Lebesgue measure on symmetric matrices, of the form $f(X) = c_N \exp(-N\mathrm{Tr}(V(X)))$ for some sufficiently nice potential function $V\colon\R\to\R$.
\end{itemize}
(There are also corresponding complex Hermitian ensembles, and a wilder world of non-Hermitian, generally non-normal, matrix ensembles; we will restrict the present discussion to the real symmetric cases.)  Both of these are natural generalizations of the {\em Gaussian Orthogonal Ensemble} $\mathrm{GOE}_N$, which is (up to a different scaling on the diagonal) a Wigner ensemble with Gaussian entries, and also an invariant ensemble with potential function $V(X) = X^2$.

Given a symmetric matrix $X_N$, enumerate its eigenvalues $\lambda_1^N\le\cdots\le\lambda_N^N$ in nondecreasing order.  The {\em empirical spectral distribution (ESD)} of $X_N$ is the random point measure
\begin{equation} \label{e.ESD} \mu_N = \frac1N\sum_{j=1}^N \delta_{\lambda^N_j}. \end{equation}
Integrating $\mu_N$ against a step function produces a histogram of the eigenvalues of $X_N$; in general, the random variables $\int f\,d\mu_N$ for test functions $f\colon\R\to\R$ are called {\em linear statistics} of the eigenvalues.  Wigner's original paper \cite{Wi55,Wi58} showed that, for the $\mathrm{GOE}_N$, the ESD $\mu_n$ converges weakly in expectation to what is now called Wigner's semicircle law: $\sigma(dx) = \frac{1}{2\pi}\sqrt{(4-x^2)_+}\,dx$.  To be precise: this means that $\E(\int f\,d\mu_N) \to \int f\,d\sigma$ for each $f\in C_b(\R)$.  This convergence was later upgraded to weak a.s.\ convergence.  Many more results are known about the fluctuations of $\mu_N$, the spacing between eigenvalues, and the distribution and fluctuations of the largest eigenvalue.

Generally speaking, the universality program's aim is to show that all asymptotic statistics of the $\mathrm{GOE}_N$ have the same behavior for {\em all} Wigner / invariant ensembles (subject to sharp technical constraints on the distributions / potentials involved).  The reader may consult the book \cite{AGZ} and its extensive bibliography for more on this endeavor, which is now largely complete.

There is also a vast literature on {\em band matrices}.  These are random matrix ensembles generalizing Wigner matrices, where the upper-triangular entries are still independent, but need not be identically distributed (so long as they satisfy some form of uniform regularity).  There is a vast literature on band matrices; see, for example, the expansive paper \cite{AZ-PTRF} which uses combinatorial and probabilistic methods to establish that a large class of band matrices have ESD converging to the semicircle law, with Gaussian fluctuations of a similar form to Wigner matrices.

There are very few papers, however, dealing with random matrices with {\em correlated entries}.  The best previously known results are in \cite{Shlyakh-IMRN,Bryc-Speicher-IEEE}, dealing with {\em block matrices}: ensembles $X_{kN}$ possessed of $k\times k$ blocks that have a fixed covariance structure (uniform among the blocks), where the $N^2$ blocks are independent up to symmetry.  (The actual ensembles studied in \cite{Shlyakh-IMRN,Bryc-Speicher-IEEE} are presented in a different form, with an overall $k\times k$ block structure with $N\times N$ blocks all whose entries are independent; this is just a basis change from the description above.)  In that setting, the tools of {\em operator-valued free probability} come to bear giving a tractable combinatorial method to analyze the asymptotic linear statistics of the eigenvalues.  These definitely do not fit the universality mould: the limiting ESD is typically not semicircular.  The combinatorial methods used to analyze such ensembles do not easily extend beyond the case that $k$ is fixes as $N\to\infty$.

Our main theorem is a significant generalization of ESD convergence for block-type matrices, both in terms of allowing $k$ to grow (slowly) with $N$, and vastly softening the rigid structure of the partition into independent blocks.

\begin{theorem} \label{t.main} Let $X_N$ be an $N\times N$ random matrix.  Assume that the entries of $X_N$ satisfy the following conditions.
\begin{itemize}
\item[(1)] The family $\{N[X_N]_{ij}^2\}_{N\in\N,1\le i,j\le N}$ is uniformly integrable.
\item[(2)] For each $N$, there is a set partition $\Pi_N$ of $\{(i,j)\colon 1\le i\le j\le N\}$ and a constant $d_N = o(\log N)$ such that each block of $\Pi_N$ has size $\le d_N$, and the entries $[X_N]_{ij}$ and $[X_N]_{k\ell}$ are independent if $(i,j)$ and $(k,\ell)$ are not in the same block of $\Pi_N$.
\end{itemize}
Then the empirical spectral distribution $\mu_N$ of $X_N$ converges weakly in probability to its mean:
\[ \int f\,d\mu_N - \E\left(\int f\,d\mu_N\right) \to_{\mathbb{P}} 0. \]
\end{theorem}
Condition (1) is analogous to the requirement that the second moments of the entries of $\sqrt{N}X_N$ are normalized in Wigner ensembles.  Condition (2) generalizes the independent block structure mentioned above; for example, in the ensembles treated in \cite{Shlyakh-IMRN,Bryc-Speicher-IEEE} but with $k$ allowed to grow with $N$ sub-logarithmically, one gets convergence of the ESD weakly in probability.  In particular, Theorem \ref{t.main} extends the results of those papers even in the case $k=O(1)$, since only convergence in expectation was known before.

Theorem \ref{t.main} is proved below in Section \ref{section.RM}.  The method we use to prove it involves through concentration of measure mediated by a powerful coercive inequality: a {\em log-Sobolev inequality}.  Briefly: a probability measure $\mu$ on $\R^d$ satisfies a log-Sobolev inequality with constant $c$ if
\[ \mathrm{Ent}_{\mu}(f^2) \le c\int |\nabla f|^2\,d\mu \]
for all sufficiently integrable positive functions $f$ with $\int f^2\,d\mu=1$; here $\mathrm{Ent}_\mu(g) = \int g\log g\,d\mu$ for a $\mu$-probability density $g$.  It first appeared in \cite{Stam} (in a slightly different form, written in terms of $g=f^2$, where the Dirichlet form on the right-hand-side becomes the relative Fisher information of $g$), in the context of Gaussian measures.  It was later rediscovered by Gross \cite{Gr75} who named it so.  Over the past four decades, it has played an important role probability theory, functional analysis, and differential geometry; see, for example, \cite{Ba94, Ba97, BH97, BT06, Da87, Da90, DS84, DS96, GZ03, Le96, Le01, Le03, Vi03, Ya96, Ya97, Ze92}.   There is a big industry of literature devoted to necessary and sufficient conditions for a log-Sobolev inequality to hold; cf.\ \cite{BE85, BL06, BL00,GR98, HS87}.

Many of the above applications rely on uniform concentration of measure bounds that hold for measures satisfying a log-Sobolev inequality; one nice form of these concentration inequalities is called a Herbst inequality, cf.\ \cite{GR98}.  Using the Herbst inequality, Guionnet \cite{Gu09}  gave a fundamentally new proof of Wigner's semicircle law; this proof automatically generalized to non-Gaussian ensembles whose entries satisfy a log-Sobolev inequality.  Motivated in part by this, the second author of the present paper developed a new approximation scheme, the {\em mollified log-Sobolev inequality}, in \cite{Zi13}: if $Y$ is any bounded random variable and $Z$ is a standard normal random variable independent from $Y$, then $Y+tZ$ satisfies a log-Sobolev inequality for all $t>0$, with a constant $c(t)$ that is bounded in terms of an exponential of $\|Y\|_{\infty}^2/t$.    Using this, together with a standard cutoff argument, generalized Guionnet's technique to give a fully general proof of Wigner's law for all Wigner ensembles.

Independence played a key role in this analysis, due to the fact that log-Sobolev inequalities behave well under products of measures.  In the setting of current interest where we no longer have independence, we will need a multivariate version of the mollified log-Sobolev inequality, with sufficient growth bounds on the constant.  That is our second main theorem, which is of independent interest.

\begin{theorem} \label{t.LSI} Let $\mx{Y}$ be a bounded random vector in $\R^d$, and let $\mx{Y}$ be a standard centered normal random vector in $\R^d$ (i.e.\ $\mathrm{Law}_{\mx{Z}}(d\mx{x}) = (2\pi)^{-d/2} e^{-|\mx{x}|^2/2}\,d\mx{x}$) independent from $\mx{Y}$.  For $0<t\le\||\mx{Y}|\|_{\infty}^2$, the measure $\mathrm{Law}_{\mx{Y}+t\mx{Z}}$ satisfies a log-Sobolev inequality, with constant $c(t)$ satisfying
\[ c(t) \le 289\||\mx{Y}|\|_{\infty}^2 \exp\left(20d+\frac{5\||\mx{Y}|\|_{\infty}^2}{t}\right). \]
\end{theorem}

\begin{remark} \begin{itemize}
\item[(1)] In \cite{Zi13}, the second author proved a bound of this form over $\R^1$, with a slightly smaller constant but still growing with an exponential of $1/t$ as $t\downarrow 0$.  In fact, growth of this kind is sharp and cannot be improved.%, as we show in Section \ref{???} below.
\item[(2)] We do not know if the optimal constant grows with dimension as this bound does.  Regardless, a dimension independent bound of this form would not improve our result in Theorem \ref{t.main} to the a.s.\ convergence that likely holds in general.
\item[(3)] Following \cite{Zi13}, in \cite{WW13} the authors generalized mollified log-Sobolev inequalites to $\R^d$ (and with a class of measures more general than compactly-supported), using fairly standard techniques like those we use below to prove Theorem \ref{t.LSI}.  However, they give no quantitative bounds on the log-Sobolev constant, which is crucial to our present analysis.
\end{itemize}
\end{remark}

We prove Theorem \ref{t.LSI} using the Lyapunov approach (with the standard choice of the Lyapunov function), paying careful attention to the explicit dependence of the LSI constant on the Lyapunov exponents.  In particular, this approach also gives bounds on the best constant in the Poincar\'e inequality for such mollified measures, as we show below in Section \ref{section.LSI}.

%The remainder of this paper is organized as follows. \ldots

\section{Symmetric Random Matrices with Log-Range Correlations\label{section.RM}}

\subsection{Guionnet's Approach to Wigner's Law}

Let us fix notation as in the introduction: let $X_N$ be a symmetric random $N\times N$ matrix ensemble with eigenvalues $\lambda^N_1\le\cdots\le \lambda^N_N$, and let $\mu_N$ denote the empirical spectral distribution (ESD) of \eqref{e.ESD}.  Wigner's law \cite{Wi55,Wi58} states that $\mu_N$ converges weakly a.s.\ to the semicircle law $\sigma$, in the case that the entries of $X_N$ is a $\mathrm{GOE}_N$.  Wigner's proof proceeded by the method of moments and is fundamentally combinatorial.  Analytic approaches (involving fixed point equations, complex PDEs, and orthogonal polynomials) developed over the decades, but the first truly probabilistic argument was provided by Guionnet in \cite[p.70, Thm. 6.6]{Gu09}.  The result can be stated thus.

\begin{theorem}\label{t.Guionnet}
\textnormal{(Guionnet).}  Let $X_N$ be a symmetric random matrix. If the joint law of entries of $\sqrt{N}X_N$ satisfies a log-Sobolev inequality with constant $c$, then for all $\epsilon>0$ and all Lipschitz $f:\mathbb{R}\rightarrow\mathbb{R}$,
\[
\mathbb{P}\left(\left|\int f \,d\mu_N-\mathbb{E}\left(\int f\,d\mu_N\right)\right|\geq \epsilon\right)
\leq 2\exp\left(\frac{-N^2 \epsilon^2}{4c||f||_{\mathrm{Lip}}^2}\right).
\]
\end{theorem}

In fact, in the Wigner ensembe setting, the i.i.d.\ condition means we really need only assume that the law of {\em each entry} satisfies a log-Sobolev inequality.  This is due to the following result often called {\em Segal's lemma}; for a proof, see \cite[p. 1074, Rk. 3.3]{Gr75}.

\begin{lemma}[Segal's Lemma] \label{l.Segal} Let $\nu_1,\nu_2$ be probability measures on $\mathbb{R}^{d_1}$ and $\mathbb{R}^{d_2}$, satisfying log-Sobolev inequalities with constants $c_1, c_2$, respectively. Then the product measure measure $\nu_1\otimes\nu_2$ on $\mathbb{R}^{d_1+d_2}$ satisfies a log-Sobolev inequality with constant $\max\{c_1, c_2\}$.
\end{lemma}

Theorem \ref{t.Guionnet} explicitly gives weak convergence in probability of $\mu_N$ to its limit mean.  Moreover, in the Wigner ensemble case where the constant $c$ is determined by the common law of the entries and so doesn't depend on $N$, the rate of convergence is fast enough that a standard Borel--Cantelli argument immediately upgrades this to a.s.\ convergence.  In \cite{Zi13}, the second author showed that, under certain integrability conditions, the empirical law of eigenvalues $\mu_N$ converges weakly in probability to its mean, {\em regardless} of whether or not the joint laws of entries satisfy a log-Sobolev inequality.  The idea is to use the mollified log-Sobolev inequality (the $d=1$ case of Theorem \ref{t.LSI} applied) to a cutoff of $X_N$ with $\mathrm{GOE}_N$ noise added in with variance $t$, and then let $t\downarrow 0$.  The explosion of the constant $c(t)$ is too fast to allow for this argument to yield a.s.\ convergence, but it still manages convergence in probability in complete generality.

For our present purposes, where we no longer assume independence or identical distribution of the entries of $X_N$,  it will not suffice to assume each entry satisfies a (mollified) log-Sobolev inequality, which is why we state Guionnet's result as such in Theorem \ref{t.Guionnet}.  Guionnet proved the theorem from the Herbst concentration inequality \cite{GR98}, which shows that Lipschitz functionals of a random variable whose law satisfies a log-Sobolev inequality have sub-Gaussian tails (with dimension-independent bounds determined by the Lipschitz norm of the functional).  Theorem \ref{t.Guionnet} is then proved by combining this with Lipschitz functional calculus, together with the following lemma from matrix theory (see \cite[p.37, Thm. 1, and p.39, Rk. 2]{HW53}).

\begin{lemma}\label{lem:hw}
\textnormal{(Hoffman, Wielandt).} Let $A, B$ be symmetric $N\times N$ matrices with eigenvalues $\lambda^A_1\leq\lambda^A_2\leq\ldots\leq\lambda^A_N$ and $\lambda^B_1\leq\lambda^B_2\leq\ldots\leq\lambda^B_N$. Then
\[
\sum_{j=1}^N (\lambda^A_j-\lambda^B_j)^2\leq\mathrm{Tr}[(A-B)^2].
\]
\end{lemma}

\subsection{The Proof of Theorem \ref{t.main}} We now proceed to prove Theorem \ref{t.main}, using Theorem \ref{t.LSI}.  Let $X_N$ be the matrix ensemble satisfying conditions (1) and (2) of Theorem \ref{t.main}.  Denote the blocks of given partition as $\Pi_N = \{P_1,\ldots,P_r\}$; .  We also make the initial restricted assumption that the entries of $\sqrt{N}X_N$ are bounded by some uniform constant $R$, $\|\sqrt{N}[X_N]_{ij}\|_\infty \le R$ for all $N$ and all $1\le i,j\le N$; we will remove this assumption at the end of the proof.

Now, let $t=t_N>0$ (to be chosen later), and let $G_N$ be a $\mathrm{GOE}_N$ (with entries of variance $\frac1N$) independent from $X_N$.  Set \begin{equation} \label{e.tildeX1} \widetilde{X}_N = X_N+t G_N. \end{equation}
For $1\le k\le r$, let $\mx{Y}_k$ denote the random vector in $\R^{|P_k|}$ given by the entries $[X_N]_{ij}$ with $(i,j)\in P_k$; similarly, let $\mx{Z}_k$ be the corresponding entries of $G_N$.  Notice that $\sqrt{N}\mx{Y}_k$ is a bounded random vector: by assumption, all of its entries have $L^\infty$-norm $\le R$, and so $\||N|\mx{Y}_k|\|_\infty^2 \le R|P_k|^{1/2} \le Rd_N^{1/2}$.  The vector $\sqrt{N}\mx{Z}_k$ is a standard normal random vector in $\R^{|P_k|}$.  Thus, by Theorem \ref{t.LSI}, the law of $\sqrt{N}(\mx{Y}_k+t\mx{Z}_k)$ satisfies a log-Sobolev inequality with constant
\begin{equation} \label{e.c(t).est.1} c(t) \le 289(Rd_N^{1/2})^2\exp\left(20d_N+\frac{5(Rd_N^{1/2})^2}{t}\right) \le 289R^2\exp\left(21d_N+\frac{5R^2d_N}{t}\right) \end{equation}
(where we have made the blunt estimate $d_N\le \exp d_N$).  By assumption, the random variables $\{\mx{Y}_k\}_{k=1}^r$ are independent, as are $\{\mx{Z}_k\}_{k=1}^r$.  Hence $\{\sqrt{N}(\mx{Y}_k+t\mx{Z}_k)\}_{k=1}^r$ are independent.  Thus, the joint law of entries of $\sqrt{N}\widetilde{X}_N$ is the product measure of the laws of these random variables.  As all their laws satisfy log-Sobolev inequalities with the same constant $c(t)$ in \eqref{e.c(t).est.1}, Segal's Lemma \ref{l.Segal} shows that:
\begin{corollary} \label{c.LSI} The joint law of entries of $\sqrt{N}\widetilde{X}_N$ satisfies a log-Sobolev inequality with constant $c(t)$ of \eqref{e.c(t).est.1}.
\end{corollary}
In particular, Guionnet's Theorem \ref{t.Guionnet} shows that the (Lipschitz) linear statistics of the ensemble $\widetilde{X}_N$ are highly concentrated around their means (for fixed $t$).

Our goal is now to compare the linear statistics of $X_N$ to those of $\widetilde{X}_N$.  As usual, let $\mu_N$ denote the ESD of $X_N$, and let $\widetilde\mu_N$ denote the ESD of $\widetilde{X}_N$.  Then, for each $\e>0$, and each test function $f$, we have the following standard $\e/3$-type estimate.
\begin{equation}\label{eqn:epsover3}
\begin{aligned}
\mathbb{P}\left(\left|\int f\,d\mu_N-\mathbb{E}\left(\int f\,d\mu_N\right)\right|\geq \epsilon\right)
&\leq\mathbb{P}\left(\left|\int f\,d\mu_N-\int f\, d\widetilde\mu_N\right|\geq \frac{\epsilon}{3}\right)\\
&+\mathbb{P}\left(\left|\int f\,d\widetilde\mu_N-\mathbb{E}\left(\int f\,d\widetilde\mu_N\right)\right|\geq \frac{\epsilon}{3}\right)\\
&+\mathbb{P}\left(\left|\mathbb{E}\left(\int f\, d\widetilde\mu_N\right)-\mathbb{E}\left(\int f\, d\mu_N\right)\right|\geq \frac{\epsilon}{3}\right).
\end{aligned}
\end{equation}
We will now show that, with a judicious choice of $t=tN$, each of the three terms in \eqref{eqn:epsover3} converges to $0$ as $N\to\infty$,
which proves the desired convergence in probability of $\mu_N$ (under the boundedness assumption).  We proceed to do this in the following three lemmas.  (Let us note that the arguments here are very similar to those in the second author's paper \cite[Lemmas 11-13]{Zi13}, and, in turn, also similar to the methodology in \cite{Gu09}.)

\begin{lemma}\label{lem:epsover31} Let $f\in\mathrm{Lip}(\R)$, and let $\e>0$.  Then for all $N\in\N$,
\[
\mathbb{P}\left(\left|\int f \,d\mu_N-\int f\,d\widetilde{\mu}\right|\geq \frac{\epsilon}{3}\right)
\leq\frac{9 \|f\|_{\mathrm{Lip}}^2}{\epsilon^2}\,t.
\]
\end{lemma}

\begin{proof}
Let $\lambda^N_1\leq\lambda^N_2\leq\ldots\leq\lambda^N_N$ and $\widetilde{\lambda}^N_1\leq\widetilde{\lambda}^N_2\leq\ldots\leq\widetilde{\lambda}^N_N$ be the eigenvalues of $X_N$ and $\widetilde{X}_N$. Then by the Cauchy-Schwarz inequality and Lemma \ref{lem:hw},
\begin{align*}
\left|\int f \, d\mu_N-\int f \, d\widetilde\mu\right|
\leq \frac{1}{N}\sum_{i=1}^N \|f\|_{\mathrm{Lip}}\left|\lambda^N_i-\widetilde{\lambda}^N_i\right|
&\leq \frac{\|f\|_{\mathrm{Lip}}}{\sqrt{N}}\left(\sum_{i=1}^N (\lambda^N_i-\widetilde{\lambda}^N_i)^2\right)^{1/2}\\
&\leq \frac{\|f\|_{\mathrm{Lip}}}{\sqrt{N}}\left(\mathrm{Tr}[(X_N-\widetilde{X}_N)^2]\right)^{1/2}.
\end{align*}
Thus
\begin{align*}
\mathbb{P}\left(\left|\int f\,d\mu_N-\int f\, d\widetilde\mu_N\right|\geq \frac{\epsilon}{3}\right)
&\leq \mathbb{P}\left(\frac{\|f\|_{\mathrm{Lip}}}{\sqrt{N}}\left(\mathrm{Tr}[(X_N-\widetilde{X}_N)^2]\right)^{1/2}\geq \frac{\epsilon}{3}\right)\\
&=\mathbb{P}\left(\mathrm{Tr}[(X_N-\widetilde{X}_N)^2]\geq \frac{\epsilon^2 N}{9 \|f\|_{\mathrm{Lip}}^2}\right).
\end{align*}
By Markov's inequality, this is bounded above by
\begin{align*}
\frac{9 \|f\|_{\mathrm{Lip}}^2}{\epsilon^2 N}\mathbb{E}\left(\mathrm{Tr}[(X_N-\widetilde{X}_N)^2]\right)
=\frac{9 \|f\|_{\mathrm{Lip}}^2}{\epsilon^2 N}\sum_{1\leq i,j \leq N}\mathbb{E}\left(([X_N]_{ij}-[\widetilde{X}_N]_{ij})^2\right)
=\frac{9 \|f\|_{\mathrm{Lip}}^2}{\epsilon^2}\,t
\end{align*}
concluding the proof.  \end{proof}

\begin{lemma}\label{lem:epsover32}  Let $f\in \mathrm{Lip}(\R)$, and let $\e>0$.  Let $c(t)$ denote the log-Sobolev constant in \eqref{e.c(t).est.1}.  Then for all $N$,
\[
\mathbb{P}\left(\left|\int f \, d\widetilde\mu_N-\mathbb{E}\left(\int f\, d\widetilde\mu_N\right)\right|\geq \frac{\epsilon}{3}\right)
\leq 2\exp\left(\frac{-N^2 \epsilon^2}{36 c(t)\|f\|_{\mathrm{Lip}}^2}\right).
\]
\end{lemma}

\begin{proof}
This is immediate from Theorem \ref{t.Guionnet} and Corollary \ref{c.LSI}.
\end{proof}

The final term in \eqref{eqn:epsover3} is the probability of a deterministic event, so it is either $0$ or $1$.  By letting $t=t_N$ shrink to $0$, the probability will be $0$ eventually.

\begin{lemma}\label{lem:epsover33}
Let $f\in\mathrm{Lip}(\R)$, and let $\e>0$.  If $t_N\to0$ as $N\rightarrow\infty$, then
\[
\mathbb{P}\left(\left|\mathbb{E}\left(\int f\,d\widetilde\mu_N\right)-\mathbb{E}\left(\int f\,d\mu_N\right)\right|\geq \frac{\epsilon}{3}\right) = 0
\]
for all sufficiently large $N$.
\end{lemma}

\begin{proof} It suffices to show that
$\left|\mathbb{E}\left(\int f\, d\widetilde\mu_N\right)-\mathbb{E}\left(\int f \,d\mu_N\right)\right|$
converges to 0 as $N\rightarrow\infty$. Doing similar estimates as in Lemma \ref{lem:epsover31}, we get
\begin{align*}
\left|\mathbb{E}\left(\int f\, d\widetilde\mu_N\right)-\mathbb{E}\left(\int f\,d\mu_N\right)\right|
\leq\mathbb{E}\left(\left|\int f\,d\widetilde\mu_N-\int f\,d\mu_N\right|\right)
&\leq \mathbb{E}\left(\frac{\|f\|_{\mathrm{Lip}}}{\sqrt{N}}\left(\mathrm{Tr}[(X_N-\widetilde{X}_N)^2]\right)^{1/2}\right)\\
&\leq \frac{\|f\|_{\mathrm{Lip}}}{\sqrt{N}}\left(\mathbb{E}\left(\mathrm{Tr}[(X_N-\widetilde{X}_N)^2]\right)\right)^{1/2}\\
&=\|f\|_{\mathrm{Lip}}t_N^{1/2}.
\end{align*}
(The last inequality following from H\"{o}lder's inequality applied to $\left(\mathrm{Tr}[(X_N-\widetilde{X}_N)^2]\right)^{1/2}$
and the constant function $1$.) The result follows.  \end{proof}

We can now prove the theorem under the boundedness assumption.

\begin{proof}[Proof of Theorem \ref{t.main} Assuming $\sqrt{N}X_N$ has uniformly bounded entries]  For $N$ sufficiently large, we define
\[
t_N:=\frac{5R^2d_N}{\log \frac{N}{289R^2}-21d_N}.
\]
By Assumption (2) of Theorem \ref{t.main}, $d_N = o(\log N)$, and hence $t_N\to 0$ as $N\to\infty$.  Note, from \eqref{e.c(t).est.1}, that
\[ c(t_N) \le 289R^2\exp\left(21d_N+\frac{5R^2d_N}{t_N}\right) \le N. \]
Applying Lemmas \ref{lem:epsover31}, \ref{lem:epsover32}, and \ref{lem:epsover33} to (\ref{eqn:epsover3}), we get that for sufficiently large $N$,
\begin{align*}
\mathbb{P}\left(\left|\int f \,d\mu_N-\mathbb{E}\left(\int f \, d\mu_N\right)\right|\geq \epsilon\right)
&\leq \frac{9 \|f\|_{\mathrm{Lip}}^2}{\epsilon^2}\,t_N + 2\exp\left(\frac{-N^2 \epsilon^2}{36c(t_N)\|f\|_{\mathrm{Lip}}^2}\right) + 0\\
&\leq \frac{9 \|f\|_{\mathrm{Lip}}^2}{\epsilon^2}\,t_N+ 2\exp\left(\frac{-N \epsilon^2}{36\|f\|_{\mathrm{Lip}}^2}\right),
\end{align*}
and this tends to $0$ as $N\to\infty$.  Hence, we get convergence in probability with Lipschitz test functions;  it is straightforward to upgrade this to convergence in probability with respect to all $C_b(\R)$ test functions, concluding the proof.  \end{proof}

To conclude the proof, it remains only to remove the boundedness assumption on the entries of $\sqrt{N}X_N$.  This is where the uniform integrability comes in, via a standard cutoff argument that we briefly outline.  Let $\epsilon,\eta>0$.  Let $f\in\mathrm{Lip}(\R)$. By uniform integrability, there exists some $R\geq0$ such that
\[
\mathbb{E}\left(N[X_N]_{ij}^2\cdot\mathds{1}_{\{\sqrt{N}|[X_N]_{ij}|>R\}}\right)<\min(1,\eta)\cdot\epsilon^2/(9||f||_{\mathrm{Lip}}^2)
\]
for all $i,j,N$. Let $\widehat{X}_N$ be the matrix whose entries are the appropriate cutoffs of $X_N$:
\[
[\widehat{X}_N]_{ij}=[X_N]_{ij}\cdot\mathds{1}_{\{\sqrt{N}|[X_N]_{ij}|\le R\}}.
\]
Then $\|\sqrt{N}\widehat{X}_{ij}\|_\infty \le R$ for all $N,i,j$.  Let $\widehat{\mu}_N$ denote the ESD of $\widehat{X}_N$.  The preceding proof shows that $\int f\,d\widehat{\mu}_N$ converge to its mean in probability.  We now compare the linear statistics of $\mu_N$ and $\widehat{\mu}_N$.  This is similar to the preceding analysis.  We make the standard $\e/3$-decomposition:
\begin{equation}\label{eq:otherepsover3}
\begin{aligned}
\mathbb{P}\left(\left|\int f \,d\mu_N-\mathbb{E}\left(\int f\,d\mu_N\right)\right|\geq \epsilon\right)
&\leq\mathbb{P}\left(\left|\int f\,d\mu_N-\int f\,d\widehat{\mu}_N\right|\geq \frac{\epsilon}{3}\right)\\
&+\mathbb{P}\left(\left|\int f\, d\widehat{\mu}_N-\mathbb{E}\left(\int f\, d\widehat{\mu}_N\right)\right|\geq \frac{\epsilon}{3}\right)\\
&+\mathbb{P}\left(\left|\mathbb{E}\left(\int f \, d\widehat{\mu}_N\right)-\mathbb{E}\left(\int f \, d\mu_N\right)\right|\geq \frac{\epsilon}{3}\right).
\end{aligned}
\end{equation}

The above proof in the uniform bounded case shows that the second term in \eqref{eq:otherepsover3} converges to $0$ as $N\to\infty$.  The first term on the right hand side of (\ref{eq:otherepsover3}) is bounded using the same reasoning as done in the proof of Lemma \ref{lem:epsover31}:
\begin{align*}
\mathbb{P}\left(\left|\int f \, d\mu_N-\int f \, d\widehat{\mu}_N\right|\geq \frac{\epsilon}{3}\right)
&\leq\frac{9 \|f\|_{\mathrm{Lip}}^2}{\epsilon^2 N}\sum_{1\leq i,j \leq N}\mathbb{E}\left(([X_N]_{ij}-[\widehat{X}_N]_{ij})^2\right)\\
&=\frac{9 \|f\|_{\mathrm{Lip}}^2}{\epsilon^2 N}\sum_ {1\leq i,j \leq N}\mathbb{E}\left([X_N]_{ij}^2\cdot\mathds{1}_{\{\sqrt{N}|[X_N]_{ij}|>R\}}\right) <\eta.
\end{align*}

Finally, the third term is bounded as in Lemma \ref{lem:epsover33}:
\begin{align*}
\left|\mathbb{E}\left(\int f\,d\widehat{\mu}_N\right)-\mathbb{E}\left(\int f\,d\mu_N\right)\right|
&\leq \frac{\|f\|_{\mathrm{Lip}}}{\sqrt{N}}\left(\mathbb{E}\left(\mathrm{Tr}[(X_N-\widehat{X}_N)^2]\right)\right)^{1/2}\\
&= \frac{\|f\|_{\mathrm{Lip}}}{\sqrt{N}}\left(\sum_{1\leq i,j \leq N}\mathbb{E}\left([X_n]_{ij}^2\cdot\mathds{1}_{\{\sqrt{N}|[X_N]_{ij}|> R\}}\right)\right)^{1/2}
<\frac{\epsilon}{3},
\end{align*}
so $\mathbb{P}\left(\left|\mathbb{E}\left(\int f\,d\widehat{\mu}_N\right)-\mathbb{E}\left(\int f\, d\mu_N\right)\right|\geq \frac{\epsilon}{3}\right)=0$.  Therefore  
\[
\limsup_{N\rightarrow\infty}\,\mathbb{P}\left(\left|\int f\,d\mu_N-\mathbb{E}\left(\int f\,d\mu_N\right)\right|\geq \epsilon\right)\leq\eta.
\]
Since $\eta>0$ was arbitrary, we have $\mathbb{P}\left(\left|\int f\, d\mu_N-\mathbb{E}\left(\int f \, d\mu_N\right)\right|\geq \epsilon\right)\rightarrow0$ as $N\rightarrow\infty$, giving convergence in probability.

\section{Mollified Log-Sobolev Inequalities on $\R^d$\label{section.LSI}}

In this section we will prove Theorem \ref{t.LSI}.  For convenience, we restate it below as Theorem \ref{thm:lsiboundRn}, in measure theoretic language.

\begin{theorem}\label{thm:lsiboundRn}
Let $\mu$ be a probability measure on $\mathbb{R}^d$ whose support is contained in a ball of radius $R$, and let $\gamma_t$ be the centered Gaussian of variance $t$ with $0<t\leq R^2$, i.e., $\gamma_t(x)=(2\pi t)^{-d/2}\exp(-\frac{|x|^2}{2t})\,dx$. Then for some absolute constant $K$, the optimal log-Sobolev constant $c(t)$ for the convolution $\mu\ast\gamma_t$ satisfies
\[
c(t)\leq K\, R^2\exp\left(20d+\frac{5R^2}{t}\right).
\]
$K$ can be taken above to be $289$. 
\end{theorem}

\begin{remark} Theorem \ref{thm:lsiboundRn} is slightly more general than Theorem \ref{t.LSI}, since it only requires the support to be contained in some ball of radius $R$; by contrast, in Theorem \ref{t.LSI}, $R$ is the radius of a ball {\em centered at $0$} containing $\mathrm{supp}\,\mu$.  If we use the theorem in this form, we could actually improve Theorem \ref{t.main} by softening the requirement that the entires be uniformly square integrable, only requiring their centered versions $\sqrt{N}([X_N]_{ij}-\E([X_N]_{ij}))$ to be uniformly square integrable.  However, since any ensembles we wish to apply Theorem \ref{t.main} to must converge in expectation, this does not given any practical improvement.
\end{remark}

\subsection{The Proof of Theorem \ref{thm:lsiboundRn}}

To prove Theorem \ref{thm:lsiboundRn}, we use the following theorem (see \cite[p.288, Thm. 1.2]{CGW10}):

\begin{theorem}\label{thm:CGW}
\textnormal{(Cattiaux, Guillin, Wu).} Let $\mu$ be a probability measure on $\mathbb{R}^d$ with $d\mu(x)=e^{-V(x)}dx$ for some $V\in C^2(\mathbb{R}^d)$. Suppose the following:
\begin{enumerate}
\item\label{assum:hess}
There exists a constant $K\leq 0$ such that $\mathrm{Hess}(V)\geq K I$.

\item\label{assum:lyapunov}
There exists a $W\in C^2(\mathbb{R}^d)$ with $W\geq 1$ and constants $b,c>0$ such that
$$
t W(x)-\langle\nabla V,\nabla W\rangle(x)\leq(b-c|x|^2)W(x)
$$
for all $x\in\mathbb{R}^d$.
\end{enumerate}
Then $\mu$ satisfies a LSI.

In particular, let $r_0,b',\lambda>0$ be such that
$$
t W(x)-\langle\nabla V,\nabla W\rangle(x)\leq-\lambda W(x)+b'\mathds{1}_{B_{r_0}}
$$
where $B_{r_0}$ denotes the ball centered at $0$ of radius $r_0$ (the existence of such $r_0,b',\lambda$ is implied by Assumption \ref{assum:lyapunov}). By \cite[p.61, Thm. 1.4]{BBCG08}, $\mu$ satisfies a Poincar\'e inequality with constant $C_P$; that is, for every sufficiently smooth $g$ with $\int g\mbox{ }d\mu=0$,
$$
\int g^2d\mu\leq C_P\int|\nabla g|^2d\mu;
$$
$C_P$ can be taken to be $(1+b'\kappa_{r_0})/\lambda$, where $\kappa_{r_0}$ is the Poincar\'e constant of $\mu$ restricted to $B_{r_0}$. A bound for $\kappa_{r_0}$ is 
$$
\kappa_{r_0}\leq Dr_0^2\,\frac{\sup_{x\in B_{r_0}}p(x)}{\inf_{x\in B_{r_0}}p(x)},
$$
where $p(x)=e^{-V(x)}$ and $D$ is some absolute constant that can be taken to be $4/\pi^2$. Let
\begin{align*}
A=&\frac{2}{c}\left(\frac{1}{\epsilon}-\frac{K}{2}\right)+\epsilon\\
B=&\frac{2}{c}\left(\frac{1}{\epsilon}-\frac{K}{2}\right)\left(b+c\int|x|^2d\mu(x)\right),
\end{align*}
where $\epsilon$ is an arbitrarily chosen parameter. Then $\mu$ satisfies a LSI with constant $A+(B+2)C_P$. 
\end{theorem} 

We remark that the statement of Theorem \ref{thm:CGW} is given in \cite{CGW10} in the more general context of Riemannian manifolds. Also, the constants given above are derived in \cite{CGW10} but not presented there; for our purposes we have collected those constants and presented them here.

With the above, we now prove Theorem \ref{thm:lsiboundRn}, which we restate here for the reader's convenience. 

\begin{theorem}
Let $\mu$ be a probability measure on $\mathbb{R}^d$ whose support is contained in a ball of radius $R$, and let $\gamma_t$ be the centered Gaussian of variance $t$ with $0<t\leq R^2$, i.e., $d\gamma_t(x)=(2\pi t)^{-n/2}\exp(-\frac{|x|^2}{2t})dx$. Then for some absolute constant $K$, the optimal log-Sobolev constant $c(t)$ for $\mu*\gamma_t$ satisfies
$$
c(t)\leq K\, R^2\exp\left(20n+\frac{5R^2}{t}\right).
$$
$K$ can be taken above to be $289$. 
\end{theorem}

\begin{proof}
By translation invariance of LSI, we will assume that $\mu$ is supported in $B_R$. We will apply Theorem \ref{thm:CGW} to $\mu_t$ and compute the appropriate bounds and expressions for $K$, $W$, $b$, $c$, $r_0$, $b'$, $\lambda$, $\kappa_{r_0}$, $C_P$, $\int|x|^2d\mu_{t}(x)$, $A$, and $B$.  

To find $K, b$, and $c$, we follow the computations as done in \cite[pp. 7-8]{WW13}. Let $V(x)=\frac{x^2}{2t}$ and $V_t(x)=-\log(p_t(x))$, so $$
d\mu_t(x)=e^{-V_t(x)}dx=d(e^{-V}*\mu)(x).
$$
Also let 
$$
d\mu_x(z)=\frac{1}{p_t(x)}e^{-V(x-z)}d\mu(z),
$$
so $\mu_x$ is a probability measure for each $x\in\mathbb{R}^d$. Then for $X\in\mathbb{R}^d$ with $|X|=1$,

\begin{align*}
\mathrm{Hess}(V_t)(X,X)(x)=&\left(\int_{B_R}\nabla_X V(x-z)d\mu_x(z)\right)^2-\int_{B_R}\left(|\nabla_XV(x-z)|^2-\mathrm{Hess}(V)(X,X)(x-z)\right)d\mu_x(z)\\
=&\frac{1}{t}-\left(\int_{B_R}|\nabla_XV(x-z)|^2d\mu_x(z)-\left(\int_{B_R}\nabla_X V(x-z)d\mu_x(z)\right)^2\right)\\
&\mbox{since }\mathrm{Hess}(V)=\frac{1}{t}I.
\end{align*}
But for any $C^1$ function $f$,
\begin{align*}
\int_{B_R}f^2d\mu_x(z)-\left(\int_{B_R}f\mbox{ }d\mu_x(z)\right)^2=&\frac{1}{2}\int_{B_R\times B_R}(f(z)-f(y))^2d\mu_x(z)d\mu_x(y)\\
\leq&2R^2\sup|\nabla f|^2,
\end{align*}
so for $f=\nabla_XV$, we get
\begin{align*}
\mathrm{Hess}(V_t)(X,X)(x)\geq\frac{1}{t}-2R^2\sup|\nabla(\nabla_XV)|^2 =\frac{1}{t}-\frac{2R^2}{t^2}.
\end{align*}
So we take
$$
K=\frac{1}{t}-\frac{2R^2}{t^2}.
$$
Note $K\leq 0$ since $t\leq R^2$.

Let 
$$
W(x)=\exp\left(\frac{|x|^2}{16t}\right).
$$
Then
\begin{align*}
\frac{t W-\langle\nabla V_t,\nabla W\rangle}{W}(x)=&\frac{n}{8t}+\frac{|x|^2}{64t^2}-\frac{1}{16t}\int_{B_R}\langle x,\nabla V(x-z)\rangle d\mu_x(z)\\
=&\frac{n}{8t}+\frac{|x|^2}{64t^2}-\frac{1}{16t^2}\int_{B_R}\left(|x|^2-\langle x,z\rangle\right) d\mu_x(z)\\
\leq&\frac{n}{8t}-\frac{3|x|^2}{64t^2}+\frac{1}{16t^2}\sup_{z\in B_R}\langle x,z\rangle\\
=&\frac{n}{8t}-\frac{3|x|^2}{64t^2}+\frac{1}{16t^2}R|x|.
\end{align*}
Using $|x|\leq|x|^2/2R+R/2$ above, we get
\begin{align*}
\frac{t W-\langle\nabla V_t,\nabla W\rangle}{W}(x)\leq\frac{n}{8t}-\frac{3|x|^2}{64t^2}+\frac{1}{16t^2}R\left(\frac{|x|^2}{2R}+\frac{R}{2}\right)
=\frac{n}{8t}+\frac{R^2}{32t^2}-\frac{1}{64t^2}|x|^2,
\end{align*}
so we take
\begin{align*}
b=&\frac{n}{8t}+\frac{R^2}{32t^2},\\
c=&\frac{1}{64t^2}.
\end{align*}

Now let
\begin{align*}
r_0=&\sqrt{16nt+2R^2},\\
b'=&\frac{1}{4t}\exp\left(n+\frac{R^2}{8t}-1\right),\\
\lambda=&\frac{n}{8t}.
\end{align*}
We claim that
$$
b-c|x|^2\leq-\lambda+b'\exp\left(-\frac{|x|^2}{16t}\right)\mathds{1}_{B_{r_0}},\hspace{7mm}\mbox{i.e.,}\hspace{7mm}\frac{b+\lambda-c|x|^2}{b'}\exp\left(\frac{|x|^2}{16t}\right)\leq\mathds{1}_{B_{r_0}},
$$
so that
$$
t W(x)-\langle\nabla V,\nabla W\rangle(x)\leq-\lambda W(x)+b'\mathds{1}_{B_{r_0}}.
$$
We have
\begin{align*}
\frac{b+\lambda-c|x|^2}{b'}\exp\left(\frac{|x|^2}{16t}\right)=&4t\exp\left(-n-\frac{R^2}{8t}+1\right)\left(\frac{n}{8t}+\frac{R^2}{32t^2}+\frac{n}{8t}-\frac{|x|^2}{64t^2}\right)\exp\left(\frac{|x|^2}{16t}\right)\\
=&\left(n+\frac{R^2}{8t}-\frac{|x|^2}{16t}\right)\exp\left(-\left(n+\frac{R^2}{8t}-\frac{|x|^2}{16t}\right)+1\right).
\end{align*}
For $|x|\geq r_0$, the above expression is nonpositive, and for $|x|\leq r_0$, the above expression is of the form $u e^{-u+1}$, which has a maximum value of 1, as desired.

Now we estimate $\kappa_{r_0}$ by estimating $\sup_{x\in B_{r_0}}p_t(x)$ and $\inf_{x\in B_{r_0}}p_t(x)$. For $x\in B_{r_0}$, we have
\begin{align*}
p_t(x)=\int_{B_R}(2\pi t)^{-n/2}\exp\left(-\frac{|x-y|^2}{2t}\right)d\mu(y)
\leq\int_{B_R}(2\pi t)^{-n/2}d\mu(y)
=(2\pi t)^{-n/2}
\end{align*}
and
\begin{align*}
p_t(x)=\int_{B_R}(2\pi t)^{-n/2}\exp\left(-\frac{|x-y|^2}{2t}\right)d\mu(y)
\geq&\int_{B_R}(2\pi t)^{-n/2}\exp\left(-\frac{(r_0+R)^2}{2t}\right)d\mu(y)\\
=&(2\pi t)^{-n/2}\exp\left(-\frac{(r_0+R)^2}{2t}\right),
\end{align*}
so
\begin{align*}
\kappa_{r_0}\leq Dr_0^2\,\frac{\sup_{x\in B_{r_0}}p(x)}{\inf_{x\in B_{r_0}}p(x)}
\leq Dr_0^2\exp\left(\frac{(r_0+R)^2}{2t}\right).
\end{align*}

We then take
\begin{align*}
C_P=&\frac{1+b'\kappa_{r_0}}{\lambda}\\
\leq&\frac{8t}{n}\left(1+\frac{1}{4t}\exp\left(n+\frac{R^2}{8t}-1\right)\cdot Dr_0^2\exp\left(\frac{(r_0+R)^2}{2t}\right)\right)\\
=&\frac{8t}{n}+\frac{D}{e}\left(32t+\frac{4R^2}{n}\right)\exp\left(n+\frac{R^2}{8t}+\frac{(\sqrt{16nt+2R^2}+R)^2}{2t}\right).
\end{align*}
Using $\sqrt{a}+\sqrt{b}\leq\sqrt{2(a+b)}$ and the assumptions $t\leq R^2$ and $n\geq 1$ above, we get
\begin{align*}
C_P\leq&\frac{8R^2}{1}+\frac{D}{e}\left(32R^2+\frac{4R^2}{1}\right)\exp\left(n+\frac{R^2}{8t}+\frac{\sqrt{2(16nt+2R^2+R^2)}^2}{2t}\right)\\
=&8R^2+\frac{36D}{e}R^2\exp\left(17n+\frac{25R^2}{8t}\right)\\
\leq&\left(8+\frac{36D}{e}\right)R^2\exp\left(17n+\frac{25R^2}{8t}\right).
\end{align*}

Next, we estimate $\int|x|^2d\mu_{t}(x)$:
\begin{align*}
\int_{\mathbb{R}^d}|x|^2d\mu_{t}(x)=&\int_{\mathbb{R}^d}\int_{B_R}|x|^2(2\pi t)^{-n/2}\exp\left(-\frac{|x-y|^2}{2t}\right)d\mu(y)dx\\
=&(2\pi t)^{-n/2}\int_{B_R}\int_{\mathbb{R}^d}|x+y|^2\exp\left(-\frac{|x|^2}{2t}\right)dx\,d\mu(y)\\
&\mbox{by replacing }x\rightarrow x+y\\
=&(2\pi t)^{-n/2}\int_{B_R}\int_{\mathbb{R}^d}(|x|^2+|y|^2)\exp\left(-\frac{|x|^2}{2t}\right)dx\,d\mu(y)\\
&+(2\pi t)^{-n/2}\int_{B_R}\int_{\mathbb{R}^d}2\langle x,y\rangle \exp\left(-\frac{|x|^2}{2t}\right)dx\,d\mu(y).
\end{align*}
The second integral in the last expression above equals $0$ since the integrand is an odd function of $x$. So
\begin{align*}
\int_{\mathbb{R}^d}|x|^2d\mu_{t}(x)=&(2\pi t)^{-n/2}\int_{B_R}\int_{\mathbb{R}^d}(|x|^2+|y|^2)\exp\left(-\frac{|x|^2}{2t}\right)dx\,d\mu(y)\\
\leq&(2\pi t)^{-n/2}\int_{\mathbb{R}^d}\int_{B_R}(|x|^2+R^2)\exp\left(-\frac{|x|^2}{2t}\right)d\mu(y)dx\\
=&(2\pi t)^{-n/2}\int_{\mathbb{R}^d}(|x|^2+R^2)\exp\left(-\frac{|x|^2}{2t}\right)dx\\
=&nt+R^2,
\end{align*}
the last integral computed using polar coordinates.

To get expressions for $A,B$, we choose $\epsilon=16t$; then $A,B$ satisfy
\begin{align*}
A=\frac{2}{c}\left(\frac{1}{\epsilon}-\frac{K}{2}\right)+\epsilon
=128t^2\left(\frac{1}{16t}-\left(\frac{1}{2t}-\frac{R^2}{t^2}\right)\right)+16t
=128R^2-40t
\leq128R^2
\end{align*}
and
\begin{align*}
B=\frac{2}{c}\left(\frac{1}{\epsilon}-\frac{K}{2}\right)\left(b+c\int|x|^2d\mu_t(x)\right)
\leq&128t^2\left(\frac{1}{16t}-\left(\frac{1}{2t}-\frac{R^2}{t^2}\right)\right)\left(\frac{n}{8t}+\frac{R^2}{32t^2}+\frac{1}{64t^2}\left(nt+R^2\right)\right)\\
=&\frac{18nR^2}{t}+\frac{6R^4}{t^2}-\frac{63n}{8}-\frac{21R^2}{8}\\
\leq&\frac{18nR^2}{t}+\frac{6R^4}{t^2}-2.
\end{align*}

Putting everything together, we get that the optimal log-Sobolev constant $c(t)$ for $\mu_t$ satisfies
\begin{align*}
c(t)\leq&A+(B+2)C_P\\
\leq&128R^2+\left(\frac{18nR^2}{t}+\frac{6R^4}{t^2}-2+2\right)\left(8+\frac{36D}{e}\right)R^2\exp\left(17n+\frac{25R^2}{8t}\right)\\
=&128R^2+12\cdot\frac{R^2}{2t}\left(3n+\frac{R^2}{t}\right)\left(8+\frac{36D}{e}\right)R^2\exp\left(17n+\frac{25R^2}{8t}\right).
\end{align*}
Applying $u\leq e^u$ to two of the terms in the expression above, we get
\begin{align*}
c(t)\leq&128R^2+12\exp\left(\frac{R^2}{2t}\right)\exp\left(3n+\frac{R^2}{t}\right)\left(8+\frac{36D}{e}\right)R^2\exp\left(17n+\frac{25R^2}{8t}\right)\\
=&128R^2+\left(96+\frac{432D}{e}\right)R^2\exp\left(20n+\frac{37R^2}{8t}\right)\\
\leq&\left(128+96+\frac{432D}{e}\right)R^2\exp\left(20n+\frac{5R^2}{t}\right)\\
\leq&289R^2\exp\left(20n+\frac{5R^2}{t}\right).
\end{align*}

This concludes the proof of Theorem \ref{thm:lsiboundRn}.
\end{proof}

%We conjecture that the optimal upper bound for $c(t)$ is independent of $d$; see Example \ref{ex:2ptRn} and the remark following that example.

\end{document}